\documentclass[10pt,leqno,letterpaper]{article}
\usepackage[utf8]{inputenc}
\usepackage{amsmath}
\usepackage{amsfonts}
\usepackage{amssymb}
\usepackage{graphicx}
\usepackage{mathrsfs}
\usepackage{upref,amsthm,amsxtra,exscale}
\usepackage{cite}
\usepackage[colorlinks=true,urlcolor=blue,
citecolor=red,linkcolor=blue,linktocpage,pdfpagelabels,
bookmarksnumbered,bookmarksopen]{hyperref}

\newtheorem{theorem}{Theorem}[section]

\newtheorem{remark}[theorem]{Remark}
\newtheorem{lemma}[theorem]{Lemma}
\newtheorem{proposition}[theorem]{Proposition}

\numberwithin{equation}{section}

\def\r{\mathbb{R}}

\def\eps{\varepsilon}

\def\tilde{\widetilde}

\def\cJ{\mathcal{J}}

\def\cN{\mathcal{N}}

\def\cU{\mathcal{U}}

\author{Mónica Clapp\footnote{M. Clapp was partially supported by UNAM-DGAPA-PAPIIT grant IN100718 (Mexico), CONACYT grant A1-S-10457 (Mexico), and Stockholm University (Sweden).}\qquad and \qquad Andrzej Szulkin\addtocounter{footnote}{+5}\footnote{A. Szulkin was partially supported by a grant from the Magnuson foundation at the Swedish Academy of Sciences.}}
\title{A simple variational approach to weakly coupled competitive elliptic systems}
\date{}

\begin{document}
\maketitle

\begin{abstract}
The main purpose of this paper is to exhibit a simple variational setting for finding fully nontrivial solutions to the weakly coupled elliptic system \eqref{eq:system}. We show that such solutions correspond to critical points of a $\mathcal{C}^1$-functional $\Psi:\mathcal{U}\to\mathbb{R}$ defined in an open subset $\mathcal{U}$ of the product $\mathcal{T}:=S_1\times\cdots\times S_M$ of unit spheres $S_i$ in an appropriate Sobolev space. We use our abstract setting to extend and complement some known results for the system \eqref{eq:system}.


\end{abstract}

\section{Introduction}

We study the weakly coupled elliptic system
\begin{equation} \label{eq:system}
\begin{cases}
-\Delta u_i + \kappa_iu_i = \mu_i|u_i|^{p-2}u_i + \sum\limits_{j\neq i}\lambda_{ij}\beta_{ij}|u_j|^{\alpha_{ij}}|u_i|^{\beta_{ij}-2}u_i, \\
u_i\in H,\qquad i,j=1,\ldots,M,
\end{cases}
\end{equation}
where $\Omega$ is a domain in $\mathbb{R}^N$, $N\geq 3$, $\mu_i>0$, $\lambda_{ij}=\lambda_{ji}<0$, $\alpha_{ij},\beta_{ij}>1$, $\alpha_{ij}=\beta_{ji}$, and $\alpha_{ij}+\beta_{ij}=p\in(2,2^*]$. As usual, $2^*:=\frac{2N}{N-2}$ is the critical Sobolev exponent. The space $H$ is, either $H^1_0(\Omega)$, or $D^{1,2}_0(\Omega)$, and the operators $-\Delta + \kappa_i$ are assumed to be well defined and coercive in $H$.

The cubic system \eqref{eq:system} in $\mathbb{R}^3$ with $\alpha_{ij}=\beta_{ij}=2$ arises as a model in many physical phenomena, for example, in the study of standing waves for a mixture of Bose-Einstein condensates of $M$-hyperfine states which overlap in space. The sign of $\mu_i$ reflects the interaction of the particles within each single state, whereas that of $\lambda_{ij}$ reflects the interaction between particles in two different states. The interaction is attractive if the sign is positive, and it is repulsive if the sign is negative. The system is called competitive if, as we are assuming here, all of the $\lambda_{ij}$'s are negative. 

A solution $u_i$ to the equation 
$$-\Delta u + \kappa_iu = \mu_i|u|^{p-2}u,\qquad u\in H,$$
gives rise to a solution of the system \eqref{eq:system} whose $i$-th component is $u_i$ and all other components are trivial, i.e., $u_j=0$ if $j\neq i$. A solution with at least one trivial and one nontrivial component is called \emph{semitrivial}. We are interested in finding solutions all of whose components are nontrivial. These are called \emph{fully nontrivial} solutions. A fully nontrivial solution is said to be \emph{positive} if every component $u_i$ is nonnegative.

The main purpose of this paper is to exhibit a simple variational setting for finding fully nontrivial solutions to the system \eqref{eq:system}. Our approach is inspired by the ideas introduced by Szulkin and Weth in \cite{sw0,sw}.

We will show that the fully nontrivial solutions to \eqref{eq:system} correspond to the critical points of a $\mathcal{C}^1$-functional $\Psi:\mathcal{U}\to\mathbb{R}$ defined in an open subset $\mathcal{U}$ of the product $\mathcal{T}:=S_1\times\cdots\times S_M$ of unit spheres $S_i$ in $H$. The functional $\Psi$ tends to infinity at the boundary of $\mathcal{U}$ in $\mathcal{T}$, thus allowing the application of the usual descending gradient flow techniques to obtain existence and multiplicity of critical points.

This variational setting can be easily extended to systems whose coefficients $\kappa_i,\mu_i,\lambda_{ij}$ are functions defined in $\Omega$ and satisfying suitable assumptions. It may also be extended, with some care, to systems having more general nonlinearities. We chose to treat only the constant coefficient system \eqref{eq:system} in order to make the ideas more transparent.

Our abstract results (Theorems \ref{thm:A} and \ref{thm:B}) apply to many interesting types of systems. Here we consider the following three. 

Firstly, we consider the subcritical system
\begin{equation} \label{eq:exterior}
\begin{cases}
-\Delta u_i + \kappa_iu_i = \mu_i|u_i|^{p-2}u_i + \sum\limits_{j\neq i} \lambda_{ij}\beta_{ij}|u_j|^{\alpha_{ij}}|u_i|^{\beta_{ij}-2}u_i, \\
u_i\in H^1_0(\Omega),\qquad i,j=1,\ldots,M,
\end{cases}
\end{equation}
with $\kappa_i,\mu_i>0$, $\lambda_{ij}=\lambda_{ji}<0$, $\alpha_{ij},\beta_{ij}>1$, $\alpha_{ij}=\beta_{ji}$, and $\alpha_{ij}+\beta_{ij}=p\in(2,2^*)$, in an exterior domain $\Omega$ of $\mathbb{R}^N$ (i.e., $\mathbb{R}^N\smallsetminus\Omega$ is bounded, possibly empty), $N\geq 3$. 

We assume that $\Omega$ is invariant under the action of a closed subgroup $G$ of the group $O(N)$ of linear isometries of $\mathbb{R}^N$, and look for $G$-invariant solutions, i.e., solutions whose components are $G$-invariant.

Let $Gx:=\{gx:g\in G\}$ denote the $G$-orbit of $x\in\mathbb{R}^N$. We prove the following result.

\begin{theorem} \label{thm:symmetric_exterior}
If $\dim(Gx)>0$ for every $x\in\mathbb{R}^N\smallsetminus\{0\}$ and $\Omega$ is a $G$-invariant exterior domain in $\mathbb{R}^N$, then the system \eqref{eq:exterior} has an unbounded sequence of $G$-invariant fully nontrivial solutions. One of them is positive and has least energy among all $G$-invariant fully nontrivial solutions.
\end{theorem}

There is an extensive literature on subcritical systems in bounded domains and in the whole of $\mathbb{R}^3$. We refer to \cite{so} for a detailed account. Theorem \ref{thm:symmetric_exterior} seems to be the first existence result for the system \eqref{eq:exterior} in an exterior domain. A cubic system of two equations with variable coefficients in an expanding exterior domain was recently considered in \cite{ll}.

Our second application concerns the critical system
\begin{equation} \label{eq:yamabe}
\begin{cases}
-\Delta u_i = \mu_i|u_i|^{2^*-2}u_i + \sum\limits_{j\neq i} \lambda_{ij}\beta_{ij}|u_j|^{\alpha_{ij}}|u_i|^{\beta_{ij}-2}u_i, \\
u_i\in D^{1,2}(\mathbb{R}^N),\qquad i,j=1,\ldots,M,
\end{cases}
\end{equation}
where $N\geq 3$, $\mu_i>0$, $\lambda_{ij}=\lambda_{ji}<0$, $\alpha_{ij},\beta_{ij}>1$, $\alpha_{ij}=\beta_{ji}$, and $\alpha_{ij}+\beta_{ij}=2^*$.

We look for solutions which are invariant under the conformal action of the group $\Gamma:=O(m)\times O(n)$ on $\mathbb{R}^N$, with $m+n=N+1$ and $n,m\geq 2$, which is induced by the isometric action of $\Gamma$ on the standard $N$-dimensional sphere, by means of the stereographic projection. We prove the following result.

\begin{theorem} \label{thm:entire_critical}
The system \eqref{eq:yamabe} has an unbounded sequence of $\Gamma$-invariant fully nontrivial solutions. One of them is positive and has least energy among all $\Gamma$-invariant fully nontrivial solutions.
\end{theorem}

Theorem \ref{thm:entire_critical} extends some earlier results obtained in \cite{cf,cp} for a system of two equations; see also \cite{glw}. Existence and multiplicity results for the purely critical system in a bounded domain may be found in \cite{cf,ps,pst}. Supercritical systems were recently considered in \cite{cc}.

Finally, we consider the critical system
\begin{equation} \label{eq:bn}
\begin{cases}
-\Delta u_i + \kappa_i u_i = \mu_i|u_i|^{2^*-2}u_i + \sum\limits_{j\neq i} \lambda_{ij}\beta_{ij}|u_j|^{\alpha_{ij}}|u_i|^{\beta_{ij}-2}u_i, \\
u_i\in D^{1,2}_0(\Omega),\qquad i,j=1,\ldots,M,
\end{cases}
\end{equation}
where $\Omega$ is a bounded domain with $\mathcal{C}^2$-boundary in $\mathbb{R}^N$, $N\geq 4$, $\kappa_i\in(-\lambda_1(\Omega),0)$, $\mu_i>0$, $\lambda_{ij}=\lambda_{ji}<0$, $\alpha_{ij},\beta_{ij}>1$, $\alpha_{ij}=\beta_{ji}$, and $\alpha_{ij}+\beta_{ij}=2^*$. As usual, $\lambda_1(\Omega)$ denotes the first Dirichlet eigenvalue of $-\Delta$ in $\Omega$.  

We prove the following result.

\begin{theorem} \label{thm:bn}
Let $N\geq 4$. Assume that $\min\{\alpha_{ij},\beta_{ij}\}\geq\frac{4}{3}$ if $N=5$ and that $\alpha_{ij}=\beta_{ij}=2$ if $N=4$, for all $i,j=1,\ldots,M$. Then, the system \eqref{eq:bn} has a positive least energy fully nontrivial solution.
\end{theorem}

Note that there is no condition on $\alpha_{ij},\beta_{ij}$, other than $\alpha_{ij},\beta_{ij}>1$ and $\alpha_{ij}+\beta_{ij}=2^*$, if $N\ge 6$. 

Theorem \ref{thm:bn} extends some earlier results obtained in \cite{cz1,cz2} for a system of two equations. Multiple positive solutions were constructed in \cite{pt} when $N=4$, and the existence of infinitely many sign-changing solutions was established in \cite{llw} when $N\geq 7$ and $\alpha_{ij}=\beta_{ij}=\frac{2^*}2$; see also \cite{ppw}.

Our variational approach is based on some elementary properties of a certain function in $M$ variables, which are established in Section \ref{sec:polynomial_system}. In Section \ref{sec:variational} we introduce our variational setting and we derive some abstract results concerning the existence and multiplicity of fully nontrivial solutions to the system \eqref{eq:system}. Section \ref{sec:applications} is devoted to the proof of Theorems \ref{thm:symmetric_exterior}, \ref{thm:entire_critical} and \ref{thm:bn}. 
\section{On a function in \textit{M} variables} \label{sec:polynomial_system}

Let $J:(0,\infty)^M\to\mathbb{R}$ be the function given by
$$J(s):= \sum_{i=1}^Ma_is_i^2 - \sum_{i=1}^Mb_is_i^p + \sum_{i\neq j}d_{ij}s_j^{\alpha_{ij}}s_i^{\beta_{ij}},$$
where $s=(s_1,\ldots,s_M)$, $a_i,b_i>0$, $d_{ij}\geq 0$, $d_{ij}=d_{ji}$, $\alpha_{ij},\beta_{ij}>1$, $\alpha_{ij}+\beta_{ij}=p>2$, and $\alpha_{ji}=\beta_{ij}$. Then, for $ i=1,\ldots, M$,
\begin{align}
\partial_i J(s) & = 2a_is_i-pb_is_i^{p-1} + \sum_{j\neq i}d_{ij}\beta_{ij}s_j^{\alpha_{ij}}s_i^{\beta_{ij}-1} + \sum_{j\neq i}d_{ji}\alpha_{ji}s_i^{\alpha_{ji}-1}s_j^{\beta_{ji}}  \label{eq:partial_s} \\
& = 2a_is_i-pb_is_i^{p-1} + 2\sum_{j\neq i}d_{ij}\beta_{ij}s_j^{\alpha_{ij}}s_i^{\beta_{ij}-1}. \nonumber
\end{align}

\begin{lemma} \label{lem:maximality} 
If $pb_i > 2\sum_{j\neq i}d_{ij}\beta_{ij}$ for all $i=1,\ldots,M$, then there exist $0<r<R<\infty$ such that
\begin{equation} \label{eq:maximality}
\max_{s\in (0,\infty)^M}J(s) = \max_{s\in [r,R]^M}J(s).
\end{equation}
In particular, $J$ attains its maximum on $(0,\infty)^M$.
\end{lemma}

\begin{proof}
Fix $R>r>0$ such that, for all $i=1,\ldots,M$,
\begin{equation*} 
2a_{i}t - \left(pb_i - 2\sum_{j\neq i}d_{ij}\beta_{ij}\right)t^{p-1} < 0 \qquad \text{if } t\in[R,\infty)
\end{equation*}
and
\begin{equation*} 
2a_{i}t-pb_{i}t^{p-1}>0\qquad \text{if }t\in(0,r].
\end{equation*}
Let $s=(s_1,\ldots,s_M)\in (0,\infty)^M$. If $s_i\ge R \text{ and } s_i=\max\{s_1,\ldots,s_M\}$, we have that
\begin{equation} \label{eq:-1a}
\partial_iJ(s) \le 2a_{i}s_i - \left(pb_i - 2\sum_{j\neq i}d_{ij}\beta_{ij}\right)s_i^{p-1} < 0,
\end{equation}
whereas, if $s_i\leq r$, then
\begin{equation} \label{eq:pos1}
\partial_i J(s) \ge 2a_{i}s_i-pb_{i}s_i^{p-1}>0.
\end{equation}
Therefore \eqref{eq:maximality} holds true.
\end{proof}

\begin{lemma} \label{lem:uniqueness}
If $J$ has a critical point in $(0,\infty)^M$, then it is unique and it is a global maximum of $J$ in $(0,\infty)^M$.
\end{lemma}

\begin{proof}
Assume first that $(1,\ldots,1)$ is a critical point of $J$. Then, from \eqref{eq:partial_s} we get that
\begin{equation} \label{eq:coef}
0<2a_{i} = pb_i - 2\sum_{j\neq i}d_{ij}\beta_{ij}  \qquad \text{for all } i=1,\ldots,M.
\end{equation}
If $s=(s_1,\ldots,s_M)$ is a critical point of $J$ in $(0,\infty)^M$, then, for each $i=1,\ldots,M$, \eqref{eq:partial_s} and \eqref{eq:coef} yield
\begin{equation} \label{eq:s}
2a_i(s_i - s_i^{p-1}) = 2\sum_{j\neq i}d_{ij}\beta_{ij}(s_i^{p-1} - s_j^{\alpha_{ij}}s_i^{\beta_{ij}-1}).
\end{equation}
Arguing by contradiction, assume that $s\ne(1,\ldots,1)$. We consider two cases. Suppose first that $s_i>1$ for some $i$. We may assume without loss of generality that $s_i\ge s_j$ for all $j$. Then, the left-hand side in \eqref{eq:s} is negative whereas the right-hand side is $\ge 0$. This is a contradiction. Now suppose that $s_i<1$ for some $i$. Again, we may assume that $s_i\le s_j$ for all $j$. Now the left-hand side in \eqref{eq:s} is positive while the right-hand side is not, a contradiction again. Hence $(1,\ldots,1)$ is the only critical point of $J$ in $(0,\infty)^M$. The inequalities \eqref{eq:coef} allow us to apply Lemma \ref{lem:maximality} to conclude that $(1,\ldots,1)$ is a global maximum.

Now, if $s^0=(s_1^0,\ldots,s_M^0)$ is a critical point of $J$ in $(0,\infty)^M$, then $(1,\ldots,1)$ is a critical point of
$$\bar J(s):= \sum_{i=1}^M \bar a_is_i^2 - \sum_{i=1}^M \bar b_is_i^p + \sum_{i\neq j} \bar d_{ij}s_j^{\alpha_{ij}}s_i^{\beta_{ij}},$$
where $\bar a_i:=a_i s_i^0$, $\bar b_i :=b_i (s_i^0)^{p-1}$ and $\bar d_{ij} :=d_{ij} (s_j^0)^{\alpha_{ij}}(s_i^0)^{\beta_{ij}-1}$, and the conclusion follows from the special case considered above.
\end{proof}

\begin{lemma} \label{lem:stability}
Assume that $J$ has a critical point $s^0$ in $(0,\infty)^M$. Then, for each $\eps>0$, there exists $\delta>0$ such that, if $\tilde d_{ij}\geq 0$ for all $i\neq j$ and
\begin{equation} \label{eq:delta}
\sum_{i=1}^M(|\tilde a_i-a_i|+|\tilde b_i-b_i|) + \sum_{i\neq j}|\tilde d_{ij}-d_{ij}| < \delta,
\end{equation}
then the function
$$
\tilde J(s) := \sum_{i=1}^M \tilde a_is_i^2 - \sum_{i=1}^M \tilde b_is_i^p + \sum_{i\neq j} \tilde d_{ij}s_j^{\alpha_{ij}}s_i^{\beta_{ij}},
$$
has a unique critical point $\tilde s^0$ in $(0,\infty)^M$ which is a global maximum and satisfies $|\tilde s^0 - s^0|<\eps$.
\end{lemma}

\begin{proof}
As in the proof of Lemma \ref{lem:uniqueness}, we may assume without loss of generality that $s^0=(1,\ldots,1)$. Then, \eqref{eq:coef} holds true. So, choosing $\delta>0$ small enough, we have that $\tilde a_i,\, \tilde b_i > 0$ and  $p\tilde b_i - 2\sum_{j\neq i}\tilde d_{ij}\beta_{ij} >0$ if \eqref{eq:delta} is satisfied. Thus, by Lemma \ref{lem:maximality}, $\tilde J$ has a global maximum $\tilde s^0$ in $(0,\infty)^M$ and, by Lemma \ref{lem:uniqueness}, it is the only critical point of $\tilde J$ in $(0,\infty)^M$.

Taking smaller $\delta,r>0$ and a larger $R>r$ if necessary, we have that $\tilde{J}$ satisfies the same inequalities and, therefore, $\tilde s^0\in (r,R)^M$. Since $(1,\ldots,1)$ is a strict maximum, it is easy to see that $|\tilde s^0 - (1,\ldots,1)|<\eps$, possibly after choosing a still smaller $\delta$.
\end{proof}

\section{The variational setting} \label{sec:variational}

The results of this section also apply to the case $N=1$ or 2 and $p\in(2,\infty)$.

Let $H$ be either $H^1_0(\Omega)$ or $D^{1,2}_0(\Omega)$ and, for $v,w\in H$, set
$$\langle v,w\rangle_i:=\int_\Omega(\nabla v\cdot\nabla w+\kappa_ivw) \qquad\text{and}\qquad \|v\|_i:=\left(\int_\Omega(|\nabla v|^2+\kappa_iv^2)\right)^{1/2}.$$
Since, by assumption, the operators $-\Delta + \kappa_i$ are well defined and coercive in $H$, we have that $\|\cdot\|_i$ is a norm in $H$, equivalent to the standard one.

Let $\mathcal{H}:=H^M$ with the norm
$$\|(u_1,\ldots,u_M)\|:=\left(\sum_{i=1}^M\|u_i\|_i^2\right)^{1/2},$$
and let $\mathcal{J}:\mathcal{H}\to\mathbb{R}$ be given by 
$$\mathcal{J}(u_1,\ldots,u_M) := \frac{1}{2}\sum_{i=1}^M\|u_i\|_i^2 - \frac{1}{p}\sum_{i=1}^M\int_\Omega\mu_i|u_i|^p - \frac{1}{2}\sum_{j\neq i}\int_\Omega\lambda_{ij}|u_j|^{\alpha_{ij}}|u_i|^{\beta_{ij}}.$$
This function is of class $\mathcal{C}^1$ and, since $\lambda_{ij}=\lambda_{ji}$ and $\beta_{ij}=\alpha_{ji}$, 
\begin{align*}
&\partial_i\mathcal{J}(u_1,\ldots,u_M)v=\langle u_i,v\rangle_i - \int_\Omega\mu_i|u_i|^{p-2}u_iv \\
&\qquad - \frac{1}{2}\sum_{j\neq i}\int_\Omega\lambda_{ij}\beta_{ij}|u_j|^{\alpha_{ij}}|u_i|^{\beta_{ij}-2}u_iv - \frac{1}{2}\sum_{j\neq i}\int_\Omega\lambda_{ji}\alpha_{ji}|u_i|^{\alpha_{ji}-2}u_iv|u_j|^{\beta_{ji}}\\
& = \langle u_i,v\rangle_i - \int_\Omega\mu_i|u_i|^{p-2}u_iv - \sum_{j\neq i}\int_\Omega\lambda_{ij}\beta_{ij}|u_j|^{\alpha_{ij}}|u_i|^{\beta_{ij}-2}u_iv,
\end{align*}
for each $v\in H$, $i=1,\ldots,M$. So the critical points of $\mathcal{J}$ are the solutions to the system \eqref{eq:system}. The fully nontrivial ones belong to the set
\[
\mathcal{N} := \{(u_1,\ldots,u_M)\in\mathcal{H}:u_i\neq 0, \;\partial_i\mathcal{J}(u_1,\ldots,u_M)u_i=0, \; \forall i=1,\ldots,M\}.
\]
This Nehari-type set was introduced in \cite{ctv}, and has been used in many works. Note that 
\begin{equation} \label{eq:J}
\mathcal{J}(u) = \frac{p-2}{2p}\sum_{i=1}^M\|u_i\|_i^2\qquad\text{if }u=(u_1,\ldots,u_M)\in\mathcal{N}.
\end{equation}
Given $u=(u_1,\ldots,u_M)\in\mathcal{H}$ and $s=(s_1,\ldots,s_M)\in(0,\infty)^M$, we write
$$su:=(s_1u_1,\ldots,s_Mu_M),$$
and we define $J_u:(0,\infty)^M\to\mathbb{R}$ by
$$J_u(s):=\mathcal{J}(su) = \sum_{i=1}^Ma_{u,i}s_i^2 - \sum_{i=1}^Mb_{u,i}s_i^p + \sum_{i\neq j}d_{u,ij}s_j^{\alpha_{ij}}s_i^{\beta_{ij}},$$
where
$$a_{u,i}:=\frac{1}{2}\|u_i\|_i^2,\qquad b_{u,i}:=\frac{1}{p}\int_\Omega\mu_i |u_i|^p,\qquad  d_{u,ij}:=-\frac{1}{2}\int_\Omega\lambda_{ij}|u_j|^{\alpha_{ij}}|u_i|^{\beta_{ij}}.$$
If $u_i\neq 0$ for all $i=1,\ldots,M$, then, as 
$$s_i\,\partial_i J_{u}(s)=\partial_i\mathcal{J}(su)[s_iu_i],\qquad i=1,\ldots,M,$$
we have that $s$ is a critical point of $J_u$ iff $su\in\mathcal{N}$. Define
\begin{align*}
\tilde{\mathcal{U}} :=& \{u\in\mathcal{H}: su\in\mathcal{N} \text{ for some }s\in (0,\infty)^M\} \\
 =&\{u\in (H\smallsetminus\{0\})^M:  J_u \text{ has a critical point in }(0,\infty)^M\}.
\end{align*}
By Lemma \ref{lem:uniqueness}, if $u\in (H\smallsetminus\{0\})^M$ and $J_u$ has a critical point in $(0,\infty)^M$, then this critical point is unique and it is a global maximum of $J_u$. We denote it by $s_u=(s_{u,1},\ldots,s_{u,M})$, and we define $\tilde{\mathfrak{m}}:\tilde{\mathcal{U}}\to\mathcal{N}$ by
$$\tilde{\mathfrak{m}}(u):=s_uu.$$
Then,
\begin{equation} \label{eq:maximum}
\mathcal{J}(\tilde{\mathfrak{m}}(u))=\max_{s\in (0,\infty)^M}\mathcal{J}(su).
\end{equation}
Let $S_i:=\{v\in H:\|v\|_i=1\}$, $\mathcal{T}:=S_1\times\cdots\times S_M$, $\mathcal{U}:=\tilde{\mathcal{U}}\cap\mathcal{T}$, and let $\mathfrak{m}:\mathcal{U}\to\mathcal{N}$ be the restriction of $\tilde{\mathfrak{m}}$ to $\mathcal{U}$. We write $\partial\mathcal{U}$ for the boundary of $\mathcal{U}$ in~$\mathcal{T}$.
 
\begin{proposition} \label{prop:m}
\begin{itemize}
\item[$(a)$]If $u=(u_1,\ldots,u_M)\in\mathcal{T}$ is such that $u_i$ and $u_j$ have disjoint supports for every $i\neq j$, then $u\in\mathcal{U}$. Hence $\mathcal{U}\ne\emptyset$. Moreover, $\cU$ is an open subset of $\mathcal{T}$.
\item[$(b)$]$\cU\neq \mathcal{T}$ if $-\lambda_{ij}\geq\max\{\frac{\mu_i}{\beta_{ij}},\,\frac{\mu_j}{\beta_{ji}}\}$ for some $i\neq j$.
\item[$(c)$]$\tilde{\mathfrak{m}}:\tilde{\mathcal{U}}\to\mathcal{N}$ is continuous, and $\mathfrak{m}:\mathcal{U}\to\mathcal{N}$ is a homeomorphism.
\item[$(d)$]There exists $d_0>0$ such that $\min_{i=1,\ldots,M}\|u_i\|_i\geq d_0$ if $(u_1,\ldots,u_M)\in \mathcal{N}$. Thus, $\mathcal{N}$ is a closed subset of $\mathcal{H}$.
\item[$(e)$]If $(u_n)$ is a sequence in $\cU$ such that $u_n\to u\in\partial\cU$, then $\|\mathfrak{m}(u_n)\| \to\infty$. 
\end{itemize}
\end{proposition}
 
\begin{proof}
$(a):$ Let  $u=(u_1,\ldots,u_M)\in\mathcal{T}$ be such that $u_i$ and $u_j$ have disjoint supports if $i\neq j$. Then, $d_{u,ij}=0$ for every $i\neq j$, and, setting $s_i:=(\mu_i\int_\Omega |u_i|^p)^{-1/(p-2)}$, we have that $(s_1u_1,\ldots,s_Mu_M)\in\cN$. This proves that $u\in\mathcal{U}$. Moreover, as $a_{u,i},b_{u,i},d_{u,ij}$ are continuous functions of $u$, Lemma \ref{lem:stability} implies that $\mathcal{U}$ is open.

$(b):$ We assume without loss of generality that $i=1$ and $j=2$. Let $v,v_3,\ldots,v_M\in H$ be nontrivial functions. Assume there exist $t_1,t_2>0$ such that $(t_1v,t_2v,v_3,\ldots,v_M)\in\mathcal{N}$. Then, as $\alpha_{ij}+\beta_{ij}=p$ and $\lambda_{ij}<0$ for all $i,j$, we have that
\begin{align*}
0 < t_1^2\|v\|^2 &\le\mu_1t_1^p\int_\Omega |v|^p+\lambda_{12}\beta_{12}t_2^{\alpha_{12}} t_1^{\beta_{12}}\int_\Omega |v|^p \\
&= t_1^{\beta_{12}}\int_\Omega |v|^p\,\left(\mu_1 t_1^{\alpha_{12}}+\lambda_{12}\beta_{12} t_2^{\alpha_{12}}\right), \\
0 < t_2^2\|v\|^2 &\le\mu_2t_2^p\int_\Omega |v|^p+\lambda_{21}\beta_{21}t_1^{\alpha_{21}} t_2^{\beta_{21}}\int_\Omega |v|^p \\
&= t_2^{\beta_{21}}\int_\Omega |v|^p\,\left(\mu_2 t_2^{\alpha_{21}}+\lambda_{21}\beta_{21} t_1^{\alpha_{21}}\right).
\end{align*}
Since $\lambda_{12}=\lambda_{21}$ and the right-hand sides above must be positive, we get that
\begin{equation*} 
\frac{t_1^{\alpha_{12}}}{t_2^{\alpha_{12}}}>-\lambda_{12}\frac{\beta_{12}}{\mu_1}\qquad\text{and}\qquad\frac{t_2^{\alpha_{21}}}{t_1^{\alpha_{21}}}>-\lambda_{12}\frac{\beta_{21}}{\mu_2},
\end{equation*}
which is impossible if $-\lambda_{12}\geq\max\{\frac{\mu_1}{\beta_{12}},\,\frac{\mu_2}{\beta_{21}}\}$.  So, if this last inequality holds true, then
\begin{equation} \label{eq:aa}
\left(\frac{v}{\|v\|_1},\frac{v}{\|v\|_2},\frac{v_3}{\|v_3\|_3},\ldots,\frac{v_M}{\|v_M\|_M}\right)\in\mathcal{T}\smallsetminus\mathcal{U}.
\end{equation}

$(c):$ If $(u_n)$ is a sequence in $\tilde{\mathcal{U}}$ and $u_n\to u\in\tilde{\mathcal{U}}$, then, for each $i,j=1,\ldots,M$ with $i\neq j$, we have that $a_{u_n,i}\to a_{u,i}$, \,$b_{u_n,i}\to b_{u}$ and $d_{u_{n},ij}\to d_{u,ij}$. So, from Lemma \ref{lem:stability} we get that $s_{u_{n},i}\to s_{u,i}$. Hence, $\tilde{\mathfrak{m}}:\tilde{\mathcal{U}}\to\mathcal{N}$ is continuous. 

The inverse of $\mathfrak{m}:\mathcal{U}\to\mathcal{N}$ is given by
\[
\mathfrak{m}^{-1}(u_1,\ldots,u_M) = \left(\frac {u_1}{\|u_1\|_1},\ldots,\frac {u_M}{\|u_M\|_M}\right),
\]
which is, obviously, continuous. 

$(d):$ If $(u_1,\ldots,u_M)\in \mathcal{N}$ then, as $\lambda_{ij}<0$ for every $i\neq j$, we have that $\|u_i\|_i^2\leq \mu_i\int_\Omega |u_i|^p$ for all $i=1,\ldots,M$. The statement now follows from Sobolev's inequality.

$(e):$ Let $(u_n)$ be a sequence in $\cU$ such that $u_n\to u\in\partial\cU$. If the sequence $(s_{u_{n},i})$ were bounded for every $i=1,\ldots,M$, then, after passing to a subsequence, $s_{u_{n},i}\to s_i$. Since $\cN$ is closed, we would have that $(s_1u_1,\ldots,s_Mu_M)\in\cN$ and, therefore, $u\in\cU$. This is impossible because $u\in\partial\cU$ and $\cU$ is open in $\mathcal{T}$. 
\end{proof}

A fully nontrivial solution $u$ to \eqref{eq:system} will be called \emph{synchronized} if $u_i=t_iv$ and $u_j=t_jv$ for some $i\ne j$ and $t_i,t_j\in\r$. 

\begin{proposition} \label{unsynchr}
There exists $\Lambda_0<0$ such that if $\lambda_{ij} < \Lambda_0$ for all $i,j$, then the system \eqref{eq:system} has no fully nontrivial synchronized solutions.
\end{proposition}

\begin{proof}
Choose $\Lambda_0$ such that $-\Lambda_0\geq\max\{\frac{\mu_i}{\beta_{ij}},\,\frac{\mu_j}{\beta_{ji}}\}$ for all $i\neq j$. Then \eqref{eq:aa} holds true and so $u$ cannot be a solution to \eqref{eq:system}. 
\end{proof}

$\mathcal{T}$ is a smooth Hilbert submanifold of $\mathcal{H}$. The tangent space to $\mathcal{T}$ at a point $u=(u_1,\ldots,u_M)\in\mathcal{T}$ is the space
$$\mathrm{T}_u(\mathcal{T}):=\{(v_1,\ldots,v_M)\in\mathcal{H}:\langle u_i,v_i\rangle_i=0\text{ for all }i=1,\ldots,M\}.$$
Let $\tilde{\Psi}:\tilde{\mathcal{U}}\to\mathbb{R}$ be given by $\tilde{\Psi}(u) := \cJ(\tilde{\mathfrak{m}}(u))$, and let $\Psi$ be the restriction of $\tilde{\Psi}$ to $\mathcal{U}$. Then,
\begin{equation} \label{eq:psi}
\Psi(u) =\frac{p-2}{2p}\sum_{i=1}^M\|s_{u,i}u_i\|_i^2=\frac{p-2}{2p}\sum_{i=1}^M s_{u,i}^2\quad\text{for every }u\in\mathcal{U}.
\end{equation}
If $u\in\mathcal{U}$ and the derivative $\Psi'(u)$ of $\Psi$ at $u$ exists, then
$$\|\Psi'(u)\|_*:=\sup\limits_{\substack{v\in\mathrm{T}_u(\mathcal{T}) \\v\neq 0}}\frac{|\Psi'(u)v|}{\|v\|},$$
i.e., $\|\cdot\|_*$ is the norm in the cotangent space $\mathrm{T}_u^*(\mathcal{T})$ to $\mathcal{T}$ at $u$.
A sequence $(u_n)$ in $\mathcal{U}$ is called a $(PS)_c$\emph{-sequence for} $\Psi$ if $\Psi(u_n)\to c$ and $\|\Psi'(u_n)\|_*\to 0$, and $\Psi$ is said to satisfy the $(PS)_c$\emph{-condition} if every such sequence has a convergent subsequence.

As usual, a $(PS)_c$\emph{-sequence for} $\mathcal{J}$ is a sequence $(u_n)$ in $\mathcal{H}$ such that $\mathcal{J}(u_n)\to c$ and $\|\mathcal{J}'(u_n)\|_{\mathcal{H}^{-1}}\to 0$, and $\mathcal{J}$ satisfies the $(PS)_c$\emph{-condition} if any such sequence has a convergent subsequence.
 
\begin{theorem} \label{thm:A}
\begin{itemize}
\item[$(i)$] $\Psi\in \mathcal{C}^1(\cU,\r)$ and
\begin{equation*}
\Psi'(u)v = \cJ'(\mathfrak{m}(u))[s_uv] \quad \text{for all } u\in\mathcal{U} \text{ and } v\in \mathrm{T}_u(\mathcal{T}).
\end{equation*}
\item[$(ii)$] If $(u_n)$ is a $(PS)_c$-sequence for $\Psi$, then $(\mathfrak{m}(u_n))$ is a $(PS)_c$-sequence for $\cJ$. Conversely, if $(u_n)$ is a $(PS)_c$-sequence for $\mathcal{J}$ and $u_n\in\mathcal{N}$ for all $n\in\mathbb{N}$, then $(\mathfrak{m}^{-1}(u_n))$ is a $(PS)_c$-sequence for $\Psi$.
\item[$(iii)$] $u$ is a critical point of $\Psi$ if and only if $\mathfrak{m}(u)$ is a fully nontrivial critical point of $\cJ$.
\item[$(iv)$] If $(u_n)$ is a sequence in $\cU$ such that $u_n\to u\in\partial\cU$, then $\Psi(u_n) \to\infty$.
\item[$(v)$] $\Psi$ is even, i.e., $\Psi(u)=\Psi(-u)$ for every $u\in\mathcal{U}$. 
\end{itemize}
\end{theorem}

\begin{proof}
We adapt the arguments of Proposition 9 and Corollary 10 in \cite{sw}.
 
$(i):$ Let $u\in\tilde{\mathcal{U}}$ and $v\in\mathcal{H}$. As $s_u$ is the maximum of $J_u$, using the mean value theorem we obtain
\begin{align*}
&\tilde\Psi(u+tv)-\tilde\Psi(u) = \cJ(s_{u+tv}(u+tv)) - \cJ(s_uu) \\
&\qquad \le  \cJ(s_{u+tv}(u+tv)) - \cJ(s_{u+tv}u) = \cJ'(s_{u+tv}(u+\tau_1 tv))\,[ts_{u+tv}v],
\end{align*}
for $|t|$ small enough and some $\tau_1\in(0,1)$. Similarly,
\[
\tilde\Psi(u+tv)-\tilde\Psi(u) \ge \cJ(s_{u}(u+tv)) - \cJ(s_{u}u) = \cJ'(s_{u}(u+\tau_2 tv))\,[ts_{u}v],
\]
for some $\tau_2\in(0,1)$. From the continuity of $s_u$ and these two inequalities we obtain
\[
\lim_{t\to 0} \frac{\tilde\Psi(u+tv)-\tilde\Psi(u)}t = \cJ'(s_uu)[s_uv] = \cJ'(\mathfrak{\tilde m}(u))[s_uv].
\]
The right-hand side is linear in $v$ and continuous in $v$ and $u$. Therefore $\tilde\Psi$ is of class $\mathcal{C}^1$. If $u\in \cU$ and $v\in \mathrm{T}_u(\mathcal{T})$, then $\mathfrak{\tilde m}(u)=\mathfrak{m}(u)$, and the statement is proved. 

$(ii):$ Note that $\mathcal{H} = \mathrm{T}_u(\mathcal{T}) \oplus(\r u_1,\ldots,\r u_M)$ for each $u\in\cU$. Since $\mathfrak{m}(u)\in\mathcal{N}$, we have that $\cJ'(\mathfrak{m}(u))w=0$ if $w\in (\r u_1,\ldots,\r u_M)$. So, from $(i)$ we get
\begin{align*}
C_0(\min_i\{s_{u,i}\})\|\mathcal{J}'(\mathfrak{m}(u))\|_{\mathcal{H}^{-1}} &\leq\|\Psi'(u)\|_* = \sup_{\substack{v\in \mathrm{T}_u(\mathcal{T}) \\ v\neq 0}}\frac{|\cJ'(\mathfrak{m}(u))[s_uv]|}{\|v\|} \\ 
&\leq (\max_i\{s_{u,i}\})\|\mathcal{J}'(\mathfrak{m}(u))\|_{\mathcal{H}^{-1}}.
\end{align*}
If $(\Psi(u_n))$ converges, then $(s_{u_n})$ is bounded in $\mathbb{R}^M$ by \eqref{eq:psi}. Moreover, by Proposition \ref{prop:m}$(d)$, this sequence is bounded away from $0$. Therefore, $(\mathfrak{m}(u_n))$ is a $(PS)_c$-sequence for $\mathcal{J}$ iff $(u_n)$ is a $(PS)_c$-sequence for $\Psi$, as claimed.

$(iii):$ As $\cJ'(\mathfrak{m}(u))w=0$ if $w\in (\r u_1,\ldots,\r u_M)$, it follows from $(i)$ that $\Psi'(u)=0$ if and only if $\cJ'(\mathfrak{m}(u))=0$. 

$(iv):$ This statement follows from Proposition \ref{prop:m}$(e)$ and \eqref{eq:J}.

$(v):$ Since $-u\in\mathcal{N}$ iff $u\in\mathcal{N}$, we have that $s_u=s_{-u}$. So, as $\mathcal{J}$ is even, $\Psi(-u)=\mathcal{J}(s_{-u}(-u))=\mathcal{J}(s_{u}u)=\Psi(u)$.
\end{proof}

Let $Z$ be a subset of $\mathcal{T}$ such that $-u\in Z$ iff $u\in Z$. If $Z\neq\emptyset$, the \emph{genus of} $Z$ is the smallest integer $k \geq 1$ such that there exists an odd continuous function $Z \to \mathbb{S}^{k-1}$ into the unit sphere $\mathbb{S}^{k-1}$ in $\mathbb{R}^k$. We denote it by $\mathrm{genus}(Z)$. If no such $k$ exists, we define $\mathrm{genus}(Z) := \infty$. We set $\mathrm{genus}(\emptyset):=0$.

As usual, we write 
$$\Psi^{\leq a}:=\{u\in\mathcal{U}:\Psi(u)\leq a\},\qquad K_c:=\{u\in\mathcal{U}:\Psi(u)=c,\;\|\Psi'(u)\|_*=0\}.$$
The previous theorem yields the following one.

\begin{theorem} \label{thm:B}
\begin{itemize}
\item[$(a)$]If $\inf_\mathcal{N}\mathcal{J}$ is attained by $\mathcal{J}$ at some $u=(u_1,\ldots,u_M)\in\mathcal{N}$, then $u$ and $|u|:=(|u_1|,\ldots,|u_M|)$ are fully nontrivial solutions of \eqref{eq:system}.
\item[$(b)$]If $\Psi:\mathcal{U}\to\mathbb{R}$ satisfies the $(PS)_c$-condition for every $c\leq a$, then  the system \eqref{eq:system} has, either an infinite (in fact, uncountable) set of fully nontrivial solutions with the same norm, or it has at least $\mathrm{genus}(\Psi^{\leq a})$ fully nontrivial solutions with pairwise different norms.
\item[$(c)$]If $\Psi:\mathcal{U}\to\mathbb{R}$ satisfies the $(PS)_c$-condition for every $c\in\mathbb{R}$ and $\mathrm{genus}(\mathcal{U})=\infty$, then the system \eqref{eq:system} has an unbounded sequence of fully nontrivial solutions.
\end{itemize}
\end{theorem}

\begin{proof}
Theorem \ref{thm:A}$(iii)$ states that $u$ is a critical point of $\Psi$ iff $\mathfrak{m}(u)$ is a fully nontrivial critical point of $\mathcal{J}$. Note that $\Psi(u)=\frac{p-2}{2p}\|\mathfrak{m}(u)\|^2$, by \eqref{eq:J}.

If $\inf_\mathcal{N}\mathcal{J}=\mathcal{J}(u)$ and $u\in\mathcal{N}$, then $\mathfrak{m}^{-1}(u)\in\mathcal{U}$ and $\Psi(\mathfrak{m}^{-1}(u))=\inf_\mathcal{U}\Psi$. So $u$ is a fully nontrivial critical point of $\mathcal{J}$. As $|u|\in\mathcal{N}$ and $\mathcal{J}(|u|)=\mathcal{J}(u)$ the same is true for $|u|$. This proves $(a)$.

Theorem \ref{thm:A}$(iv)$ implies that $\mathcal{U}$ is positively invariant under the negative pseudogradient flow of $\Psi$, so the usual deformation lemma holds true for $\Psi$; see, e.g., \cite[Section II.3]{s} or \cite[Section 5.3]{w}. Set
$$c_j:=\inf\{c\in\mathbb{R}:\mathrm{genus}(\Psi^{\leq c})\geq j\}.$$
Standard arguments show that, under the assumptions of $(b)$, $c_j$ is a critical value of $\Psi$ for every $j=1,\ldots,\mathrm{genus}(\Psi^{\leq a})$. Moreover, if some of these values coincide, say $c:=c_j=\cdots=c_{j+k}$, then $\mathrm{genus}(K_c)\geq k+1\geq 2$. Hence, $K_c$ is an infinite set; see, e.g., \cite[Lemma II.5.6]{s}. On the other hand, under the assumptions of $(c)$, $c_j$ is a critical value for every $j\in\mathbb{N}$, and a well known argument (see, e.g., \cite[Proposition 9.33]{ra}) shows that $c_j\to\infty$ as $j\to\infty$. This completes the proof.
\end{proof}

\section{Some applications} \label{sec:applications}

\subsection{Subcritical systems in exterior domains}

Consider the subcritical system \eqref{eq:exterior} in an exterior domain $\Omega$. First, we show that this system cannot be solved by minimization. Set 
$$S_{p,i}:=\inf\limits_{\substack{w\in H^1(\mathbb{R}^N) \\ w\neq 0}}\frac{\|w\|_i^2}{|w|_{p,i}^2},$$
where \, $\|w\|_i^2:=\int_{\mathbb{R}^N}(|\nabla w|^2+\kappa_iw^2)$ \, and \, $|w|_{p,i}^p:=\int_{\mathbb{R}^N}\mu_i|w|^p.$

\begin{proposition}
We have that
\begin{equation} \label{eq:infimum}
\inf_{u\in\mathcal{N}}\mathcal{J}(u)=\frac{p-2}{2p}\sum_{i=1}^M S_{p,i}^\frac{p}{p-2}
\end{equation}
and this infimum is not attained by $\mathcal{J}$ on $\mathcal{N}$.
\end{proposition}

\begin{proof}
We consider $H^1_0(\Omega)$ to be a subspace of $H^1(\mathbb{R}^N)$, via trivial extension. 

If $(u_1,\ldots,u_M)\in \mathcal{N}$ then, as $\lambda_{ij}<0$ for every $i\neq j$, we have that $\|u_i\|_i^2\leq |u_i|_{p,i}^p$ for all $i=1,\ldots,M$. Hence, 
$$S_{p,i}\leq\frac{\|u_i\|_i^2}{|u_i|_{p,i}^2}\leq (\|u_i\|_i^2)^\frac{p-2}{p}.$$
It follows from \eqref{eq:J} that $\mathcal{J}(u)\geq \frac{p-2}{2p}\sum_{i=1}^M S_{p,i}^\frac{p}{p-2}$.

To prove the opposite inequality, set $B_r(x):=\{y\in\mathbb{R}^N:|y-x|<r\}$, and let $w_{i,R}$ be a least energy solution to the problem
$$-\Delta w + \kappa_iw = \mu_i|w|^{p-2}w,\qquad w\in H_0^1(B_R(0)).$$ 
It is easy to verify that $\lim_{R\to\infty}\|w_{i,R}\|_i^2=S_{p,i}^\frac{p}{p-2}$. Fix $\xi_{i,R}\in\Omega$, $i=1\ldots,m$, such that $B_R(\xi_{i,R})\subset\Omega$ and $B_R(\xi_{i,R})\cap B_R(\xi_{j,R})=\emptyset$ if $i\neq j$, and set $u_R:=(u_{1,R},\ldots,u_{M,R})$ with $u_{i,R}(x):=w_{i,R}(x-\xi_{i,R})$. Then, $u_R\in\mathcal{N}$ and 
$$\lim_{R\to\infty}\mathcal{J}(u_R)= \frac{p-2}{2p}\sum_{i=1}^M S_{p,i}^\frac{p}{p-2}.$$
This completes the proof of \eqref{eq:infimum}.

To show that the infimum is not attained, we argue by contradiction. Assume that $(u_1,\ldots,u_M)\in \mathcal{N}$ and $\mathcal{J}(u)=\frac{p-2}{2p}\sum_{i=1}^M S_{p,i}^\frac{p}{p-2}$. We may assume that $u_i\geq 0$ for all $i=1,\ldots,M$. We fix $i$ and consider two cases. If $\int_\Omega u_j^{\alpha_{ij}}u_i^{\beta_{ij}}\neq 0$ for some $j\neq i$, then $\|u_i\|_i^2<|u_i|_{p,i}^p$ and, hence, $S_{p,i}^{p/(p-2)} < \|u_i\|_i^2$. This implies that $\mathcal{J}(u)>\frac{p-2}{2p}\sum_{i=1}^M S_{p,i}^{p/(p-2)}$, contradicting our assumption. On the other hand, if $\int_\Omega u_j^{\alpha_{ij}}u_i^{\beta_{ij}}= 0$ for all $j\neq i$, then $\|u_i\|_i^2=|u_i|_{p,i}^p=S_{p,i}^{p/(p-2)}$. Hence, $u_i$ is a nontrivial solution to the problem
$$-\Delta w + \kappa_iw = \mu_i|w|^{p-2}w,\qquad w\in H_0^1(\mathbb{R}^N).$$
Moreover, $\int_\Omega u_j^{\alpha_{ij}}u_i^{\beta_{ij}}= 0$ also implies that $u_j^{\alpha_{ij}}u_i^{\beta_{ij}}= 0$ a.e. in $\Omega$. As $u_j\not\equiv 0$ for all $j$, we have that $u_i=0$ in some subset of positive measure of $\mathbb{R}^N$. This contradicts the maximum principle.
\end{proof}

To obtain multiple solutions to the system $\eqref{eq:exterior}$ we introduce some symmetries. 

Let $G$ be a closed subgroup of $O(N)$ and $Gx:=\{gx:g\in G\}$. Set $\mathbb{S}^{N-1}:=\{x\in\mathbb{R}^N:|x|=1\}$. We start with the following lemma.

\begin{lemma} \label{lem:infinite_orbits}
If $\dim(Gx)>0$ for every $x\in\mathbb{R}^N\smallsetminus\{0\}$, then, for each $k\in\mathbb{N}$, there exists $d_k>0$ such that, for every $x\in\mathbb{S}^{N-1}$, there exist $g_1,\ldots,g_k\in G$ with
$$\min_{i\neq j}|g_ix-g_jx|\geq d_k.$$ 
\end{lemma}

\begin{proof}
Arguing by contradiction, assume that for some $k\in\mathbb{N}$ and every $n\in\mathbb{N}$ there exists $x_n\in\mathbb{S}^{N-1}$ such that
$$\min_{i\neq j}|g_ix_n-g_jx_n|<\frac{1}{n}\qquad \text{for any }k\text{ elements }\,g_1,\ldots,g_k\in G.$$
After passing to a subsequence, we have that $x_n\to x$ in $\mathbb{S}^{N-1}$. Since $\dim(Gx)>0$, there exist $\bar{g}_1,\ldots,\bar{g}_k\in G$ such that $\bar{g}_ix\neq \bar{g}_jx$ if $i\neq j$. Fix $i\neq j$ such that, after passing to a subsequence, $|\bar{g}_ix_n-\bar{g}_jx_n|=\min_{i\neq j}|\bar{g}_ix_n-\bar{g}_jx_n|$ for every $n\in\mathbb{N}$. Then, 
$$0<\min_{i\neq j}|\bar{g}_ix-\bar{g}_jx|\leq|\bar{g}_ix-\bar{g}_jx|=\lim_{n\to\infty}|\bar{g}_ix_n-\bar{g}_jx_n|=0.$$
This is a contradiction.
\end{proof}

We assume that $\Omega$ is $G$-invariant and define
$$H^1_0(\Omega)^G:=\{v\in H^1_0(\Omega):v\text{ is }G\text{-invariant}\}\qquad\text{and}\qquad\mathcal{H}^G:=(H^1_0(\Omega)^G)^M.$$
Recall that $\Omega$ is called $G$-invariant if $Gx\subset\Omega$ for all $x\in\Omega$, and a function $v:\Omega\to\mathbb{R}$ is $G$-invariant if it is constant on $Gx$ for every $x\in\Omega$. An $M$-tuple $(v_1,\ldots,v_M)$ will be called $G$-invariant if each component $v_i$ is $G$-invariant.

\begin{lemma} \label{lem:compactness_subcritical}
Assume that $\dim(Gx)>0$ for every $x\in\mathbb{R}^N\smallsetminus\{0\}$ and let $\Omega$ be a $G$-invariant exterior domain. Then, the embedding $H^1_0(\Omega)^G\hookrightarrow L^p(\Omega)$ is compact for every $p\in(2,2^*)$.
\end{lemma}

\begin{proof}
Let $(w_n)$ be a bounded sequence in $H_{0}^{1}(\Omega )^{G}$. Then, after passing to a subsequence, 
$w_n\rightharpoonup w$ weakly in $H_{0}^{1}(\Omega )^{G}$. Set $v_n:=w_n-w$. A subsequence of $(v_n)$ satisfies $v_n\rightharpoonup 0$ weakly in $H_{0}^{1}(\Omega )^{G}$, $v_n\to 0$ in $L_{loc}^{2}(\Omega)$ and $v_n(x)\to 0$ a.e. in $\Omega$. We claim that
\begin{equation}
\sup_{x\in \mathbb{R}^{N}}\int_{B_1(x)}v_{n}^{2}\to 0\qquad \text{as }\;n\to \infty. \label{eq:lions}
\end{equation}
To prove this claim, let $\varepsilon >0$, and let $C>0$ be such that $\|v_n\|^2\leq C$ for all $n\in \mathbb{N}$, where $\|\cdot\|$ is the standard norm in $H_{0}^{1}(\Omega )$. We choose $k\in \mathbb{N}$ such that $C<\varepsilon k$ and $d_k>0$ as in Lemma \ref{lem:infinite_orbits}, and we fix $R_k>2/d_k$. We consider two cases.

Assume first that $|x|\geq R_k$. By Lemma \ref{lem:infinite_orbits}, there exist $g_1,\ldots,g_k\in G$ such that
$$|g_ix-g_jx|\geq|x|d_k\quad\text{for all } i\neq j.$$
Since $|x|\geq R_k$, we have that $|g_ix-g_jx|> 2$. Hence, $B_1(g_ix)\cap B_1(g_jx)=\emptyset$ if $i\neq j$ and, as $v_n$ is $G$-invariant, we obtain
$$k\int_{B_1(x)}v_{n}^{2}=\sum\limits_{i=1}^k\int_{B_1(g_ix)}v_n^2\leq \int_{\Omega}v_n^2\leq\|v_n\|^2\leq C\qquad \text{for all }n\in \mathbb{N}.$$
Therefore,
\begin{equation}
\int_{B_1(x)}v_n^2<\varepsilon \qquad\text{for all }n\in \mathbb{N}\text{ and all }|x|\geq R_k. \label{eq:outside}
\end{equation}

Now assume that  $|x|\leq R_k$. Then, since $v_n\to 0$ strongly in $L^2(B_{R_k+1}(0))$, there exists $n_0\in \mathbb{N}$ such that
\begin{equation}
\int_{B_1(x)}v_n^2\leq \int_{B_{R_k+1}(0)}v_n^2<\varepsilon \qquad\text{for all }n\geq n_{0}.  \label{eq:inside}
\end{equation}

Inequalities \eqref{eq:outside} and \eqref{eq:inside} yield \eqref{eq:lions}. Applying Lions' lemma \cite[Lemma 1.21]{w} we conclude that $v_n\to 0$ strongly in $L^p(\Omega)$ for any $p\in (2,2^*)$.
\end{proof}

\begin{lemma} \label{lem:PS_exterior}
Assume that $\dim(Gx)>0$ for every $x\in\mathbb{R}^N\smallsetminus\{0\}$ and let $\Omega$ be a $G$-invariant exterior domain. Then, the functional $\mathcal{J}$ satisfies the Palais-Smale condition in $\mathcal{H}^G$, i.e., every sequence $(u_n)$ in $\mathcal{H}^G$ such that $\mathcal{J}(u_n)\to c$ and $\mathcal{J}'(u_n)\to 0$ in $(\mathcal{H}^G)'$, contains a convergent subsequence.
\end{lemma}

\begin{proof}
Since 
\[
\frac{p-2}p\|u_n\|^2 = \cJ(u_n)-\cJ'(u_n)u_n \le c_1+c_2\|u_n\|,
\]
$(u_n)$ is bounded. The rest of the proof follows from Lemma \ref{lem:compactness_subcritical} by standard arguments.
\end{proof}

\begin{lemma} \label{lem:genus_exterior}
Let $\mathcal{U}^G:=\mathcal{U}\cap\mathcal{H}^G$. Then, $\mathrm{genus}(\mathcal{U}^G)=\infty$.
\end{lemma}

\begin{proof}
Given $k\geq 1$, for each $j=1,\ldots,k$, $i=1,\ldots,M$, we choose $u_{j,i}\in H_0^1(\Omega)^G$ such that $\|u_{j,i}\|_i=1$ and $\mathrm{supp}(u_{j,i})\cap \mathrm{supp}(u_{j',i'})=\emptyset$ if $(j,i)\neq (j',i')$.

Let $\{ e_j : 1 \leq j \leq k\}$ be the canonical basis of $\mathbb{R}^k$, and $Q$ be the set
$$Q := \left\{\sum_{j=1}^k r_j \hat{e}_j : \hat{e}_j \in \{\pm e_j\}, \, r_j \in [0,1], \, \sum_{j=1}^k r_j = 1 \right\}.$$
Note that $Q$ is homeomorphic to the unit sphere $\mathbb{S}^{k-1}$ in $\mathbb{R}^k$ by an odd homeomorphism.

For each $i=1,\ldots,M$, define $\sigma_i:Q\to H_0^1(\Omega)^G$ by setting $\sigma_i(e_j) := u_{j,i}$, $\sigma_i(-e_j):= -u_{j,i}$, and
$$\sigma_i\left(\sum_{j=1}^k r_j \hat{e}_j \right) := \frac{\sum_{j=1}^k r_j \sigma_i(\hat{e}_j)}{\|\sum_{j=1}^k r_j \sigma_i(\hat{e}_j)\|_i}.$$
Note that, since $u_{j,i}$ and $u_{j',i'}$ have disjoint supports if $(j,i)\neq (j',i')$, these maps are well defined and $\mathrm{supp}(\sigma_i(z))\cap\mathrm{supp}(\sigma_{i'}(z))=\emptyset$ if $i\neq i'$ for every $z\in Q$. So, by Proposition \ref{prop:m}$(a)$, the map $\sigma: Q \to\mathcal{U}^G$ given by $\sigma(z):=(\sigma_1(z),\ldots,\sigma_M(z))$ is well defined. As each $\sigma_i$ is continuous and odd, so is $\sigma$. Hence, $\mathrm{genus}(\mathcal{U}^G) \geq \mathrm{genus}(Q)=k$.
\end{proof}

\begin{proof}[Proof of Theorem \ref{thm:symmetric_exterior}]
The functional $\mathcal{J}$ is $G$-invariant, so, by the principle of symmetric criticality, the critical points of the restriction of $\mathcal{J}$ to $\mathcal{H}^G$ are the $G$-invariant critical points of $\mathcal{J}$; see, e.g., \cite[Theorem 1.28]{w}.

It is readily seen that the results of Section \ref{sec:variational} are also true for $H:=H_0^1(\Omega)^G$. Theorem \ref{thm:A}$(ii)$ and Lemma \ref{lem:PS_exterior} imply that $\Psi$ satisfies the $(PS)_c$-condition for every $c\in\mathbb{R}$. This, together with Lemma  \ref{lem:genus_exterior} and Theorem \ref{thm:B}, yields Theorem \ref{thm:symmetric_exterior}.
\end{proof}

\subsection{Entire solutions to critical systems} \label{subsec:yamabe}

Next, we consider the Yamabe system \eqref{eq:yamabe}.

As usual, we denote 
$$S:=\inf\limits_{\substack{w\in D^{1,2}(\mathbb{R}^N) \\ w\neq 0}}\frac{\|w\|^2}{|w|_{2^*}^2},$$
where \, $\|w\|^2:=\int_{\mathbb{R}^N}|\nabla w|^2$ \, and \, $|w|_{2^*}^{2^*}:=\int_{\mathbb{R}^N}|w|^{2^*}$. The next result says that the system \eqref{eq:yamabe} cannot be solved by minimization. 

\begin{proposition}
We have that
$$\inf_{u\in\mathcal{N}}\mathcal{J}(u)=\frac{1}{N}\sum_{i=1}^M \mu_i^{-\frac{N-2}{2}}S^\frac{N}{2}$$
and this infimum is not attained by $\mathcal{J}$ on $\mathcal{N}$.
\end{proposition}

\begin{proof}
Following the argument given in \cite[Proposition 2.2]{cp} for $M=2$ one can easily prove this statement.
\end{proof}

To obtain multiple solutions to the system \eqref{eq:yamabe} we consider a conformal action on $\mathbb{R}^N$, as in \cite{d,cp}. 

Let $\Gamma=O(m)\times O(n)$ with $m+n=N+1$ and $m,n\geq 2$ act on $\mathbb{R}^{N+1}\equiv\mathbb{R}^m\times\mathbb{R}^n$ in the obvious way. Then, $\Gamma$ acts isometrically on the unit sphere $\mathbb{S}^N:=\{x\in\mathbb{R}^{N+1}:|x|=1\}$. The stereographic projection $\sigma:\mathbb{S}^N\to\mathbb{R}^N\cup\{\infty\}$, which maps the north pole $(0,\ldots,0,1)$ to $\infty$, induces a conformal action of $\Gamma$ on $\mathbb{R}^N$, given by 
\begin{equation} \label{eq:conformal_action}
(\gamma,x)\mapsto\tilde{\gamma}x, \qquad\text{where }\;\tilde{\gamma}:=\sigma\circ\gamma^{-1}\circ\sigma^{-1}:\mathbb{R}^N\to\mathbb{R}^N.
\end{equation}
Note that the map $\tilde{\gamma}$ is well defined except at a single point.

The group $\Gamma$ acts on the Sobolev space $D^{1,2}(\mathbb{R}^N)$ by linear isometries as follows:
$$\gamma w:=|\det \tilde{\gamma}'|^{1/2^*}w\circ\tilde{\gamma},\qquad\text{for any }\gamma\in\Gamma\;\text{ and }\;w\in D^{1,2}(\mathbb{R}^N);$$
see \cite[Section 3]{cp}. We shall say that $w$ is $\Gamma$-invariant if $\gamma w=w$ for all $\gamma\in\Gamma$, and that $(u_1,\ldots,u_M)$ is $\Gamma$-invariant if each $u_i$ is $\Gamma$-invariant. We set
$$D^{1,2}(\mathbb{R}^N)^\Gamma:=\{w\in D^{1,2}(\mathbb{R}^N):w\text{ is }\Gamma\text{-invariant}\},\qquad\mathcal{H}^\Gamma:=(D^{1,2}(\mathbb{R}^N)^\Gamma)^M.$$
One has the following results.

\begin{lemma} \label{lem:compactness_critical}
The embedding $D^{1,2}(\mathbb{R}^N)^\Gamma\hookrightarrow L^{2^*}(\mathbb{R}^N)$ is compact.
\end{lemma}

\begin{proof}
This follows from Proposition 3.3 and Example 3.4(1) in \cite{cp}.
\end{proof}

\begin{lemma} \label{lem:PS_critical}
The functional $\mathcal{J}$ satisfies the Palais-Smale condition in $\mathcal{H}^\Gamma$.
\end{lemma}

\begin{proof}
The proof follows from Lemma \ref{lem:compactness_critical} by standard arguments (boundedness of  Palais-Smale sequences is proved as in Lemma \ref{lem:PS_exterior}).
\end{proof}

\begin{lemma} \label{lem:genus_critical}
Let $\mathcal{U}^\Gamma:=\mathcal{U}\cap\mathcal{H}^\Gamma$. Then, $\mathrm{genus}(\mathcal{U}^\Gamma)=\infty$.
\end{lemma}

\begin{proof}
The proof is the same as that of Lemma \ref{lem:genus_exterior}.
\end{proof}

\begin{proof}[Proof of Theorem \ref{thm:entire_critical}]
The functional $\mathcal{J}$ is $\Gamma$-invariant; see \cite[Section 3]{cp}. Thus, the critical points of the restriction of $\mathcal{J}$ to $\mathcal{H}^\Gamma$ are the $\Gamma$-invariant critical points of $\mathcal{J}$. 

The results of Section \ref{sec:variational} hold true for $H=D^{1,2}(\mathbb{R}^N)^\Gamma$. Theorem \ref{thm:A}$(ii)$ and Lemma \ref{lem:PS_critical} imply that $\Psi$ satisfies the $(PS)_c$-condition for every $c\in\mathbb{R}$. This, together with Lemma \ref{lem:genus_critical} and Theorem \ref{thm:B}, yields Theorem \ref{thm:entire_critical}.
\end{proof}

\subsection{Brezis-Nirenberg systems}

Finally, we consider the Brezis-Nirenberg type system \eqref{eq:bn}.

For each $I\subset \{1,\ldots,M\}$, let $(\mathscr{S}_I)$ be the system of $M-|I|$ equations obtained by replacing $\kappa_i, \mu_i, \lambda_{ij},\lambda_{ji}$ with $0$ if $i\in I$, where $|I|$ is the cardinality of $I$, i.e.,
\begin{equation*}
(\mathscr{S}_I)\qquad
\begin{cases}
-\Delta u_i + \kappa_i u = \mu_i|u_i|^{2^*-2}u_i + \sum\limits_{j\neq i} \lambda_{ij}\beta_{ij}|u_j|^{\alpha_{ij}}|u_i|^{\beta_{ij}-2}u_i, \\
u_i\in D^{1,2}_0(\Omega),\qquad i,j\in\{1,\ldots,M\}\smallsetminus I.
\end{cases}
\end{equation*}
The fully nontrivial solutions of $(\mathscr{S}_I)$ correspond to the solutions $(u_1,\ldots,u_M)$ of \eqref{eq:bn} which satisfy $u_i=0$ iff $i\in I$. We set
$$c_I:=\inf\{\mathcal{J}(u):u=(u_1,\ldots,u_M)\text{ solves }\eqref{eq:bn}\text{ and }u_i=0\text{ iff }i\in I\}.$$

\begin{lemma} \label{lem:bn_compactness}
If
$$c_0:=\inf_{u\in\mathcal{N}}\mathcal{J}(u)<\min\left\{c_I + \frac{1}{N}\sum_{i\in I} \mu_i^{-\frac{N-2}{2}}S^\frac{N}{2}:\emptyset\neq I\subset\{1,\ldots,M\}\right\},$$
then this infimum is attained by $\mathcal{J}$ on $\mathcal{N}$.
\end{lemma}

\begin{proof}
Note that $\inf_{v\in\mathcal{N}}\mathcal{J}(v)=\inf_{v\in\mathcal{U}}\Psi(v)$. So, by Ekeland's variational principle \cite[Theorem 8.5]{w} and Theorem \ref{thm:A}, there exists a sequence $(u_n)$ in $\mathcal{N}$ such that $\mathcal{J}(u_n)\to c_0$ and $\mathcal{J}'(u_n)\to 0$. It follows from \eqref{eq:J} that $(u_n)$ is bounded in $\mathcal{H}:=(D^{1,2}_0(\Omega))^M$. So, after passing to a subsequence, $u_n\rightharpoonup u$ weakly in $\mathcal{H}$, $u_n\to u$ strongly in $L^2(\Omega)$ and $u_n\to u$ a.e. in $\Omega$. A standard argument shows that $u$ is a solution to the system \eqref{eq:bn}. We claim that $u$ is fully nontrivial.

Arguing by contradiction, assume that some components of $u$ are trivial. Let $I:=\{i\in\{1,\ldots,M\}:u_i=0\}$. Then, for each $i\in I$, we have that $u_{n,i}\to 0$ strongly in $L^2(\Omega)$. As $u_n\in\mathcal{N}$ and $\lambda_{ij}<0$, we get that
$$\|u_{n,i}\|^2 + o(1) = \|u_{n,i}\|^2_i \leq \mu_i|u_{n,i}|_{2^*}^{2^*},$$
where
$$\|w\|^2:=\int_{\Omega}|\nabla w|^2,\qquad |w|_p^p:=\int_{\Omega}|w|^p,\qquad\|w\|^2_i:=\|w\|^2+\kappa_i|w|_2^2.$$
Hence,
$$S\le\frac{\|u_{n,i}\|^2}{|u_{n,i}|_{2^*}^2}=\frac{\|u_{n,i}\|_i^2}{|u_{n,i}|_{2^*}^2}+o(1)\leq \mu_i^\frac{N-2}{N}(\|u_{n,i}\|_i^2)^\frac{2}{N} + o(1),$$
i.e.,\, $\mu_i^{-\frac{N-2}{2}}S^\frac{N}{2}\leq \|u_{n,i}\|_i^2 + o(1)$\, for every $i\in I$. As $u$ solves \eqref{eq:bn}, we obtain
\begin{align*}
c_0 &=\lim_{n\to\infty}\mathcal{J}(u_n)=\lim_{n\to\infty}\frac{1}{N}\left(\sum_{i\not\in I}\|u_{n,i}\|_i^2 + \sum_{i\in I}\|u_{n,i}\|_i^2\right)\\
&\geq \liminf_{n\to\infty}\frac{1}{N}\sum_{i\not\in I}\|u_{n,i}\|_i^2 + \frac{1}{N}\sum_{i\in I}\mu_i^{-\frac{N-2}{2}}S^\frac{N}{2}\\
&\geq \frac{1}{N}\sum_{i\not\in I}\|u_i\|_i^2 + \frac{1}{N}\sum_{i\in I}\mu_i^{-\frac{N-2}{2}}S^\frac{N}{2} = \mathcal{J}(u)+ \frac{1}{N}\sum_{i\in I}\mu_i^{-\frac{N-2}{2}}S^\frac{N}{2}\\
&\geq c_I + \frac{1}{N}\sum_{i\in I}\mu_i^{-\frac{N-2}{2}}S^\frac{N}{2}.
\end{align*}
This contradicts our assumption.

Therefore, $u$ is fully nontrivial. This implies that $u\in\mathcal{N}$, and \eqref{eq:J} yields
$$c_0\leq \mathcal{J}(u) \leq \liminf_{n\to\infty}\mathcal{J}(u_n)=c_0.$$
Hence, $\mathcal{J}(u)=c_0$, as claimed.
\end{proof}

\begin{lemma} \label{lem:bn_estimate}
Let $N\geq 4$. Assume that $\min\{\alpha_{ij},\beta_{ij}\}\geq\frac{4}{3}$ if $N=5$ and $\alpha_{ij}=\beta_{ij}=2$ if $N=4$, for all $i,j=1,\ldots,M$. Then
\begin{equation} \label{eq:bn_estimate}
\inf_{u\in\mathcal{N}}\mathcal{J}(u)<\min\left\{c_I + \frac{1}{N}\sum_{i\in I} \mu_i^{-\frac{N-2}{2}}S^\frac{N}{2}:\emptyset\neq I\subset\{1,\ldots,M\}\right\}.
\end{equation}
\end{lemma}

\begin{proof}
We prove this statement by induction on $M$.

If $M=1$ the system reduces to the single equation
$$-\Delta u + \kappa_1 u = \mu_1|u|^{2^*-2}u,\qquad u\in D_0^{1,2}(\Omega),$$
and the statement was proved by Brezis and Nirenberg in \cite{bn}.

Assume that the statement is true for every system $(\mathscr{S}_I)$ with $|I|\geq 1$ (i.e., for every system of $M-1$ equations). Then, the right-hand side of \eqref{eq:bn_estimate} reduces to
$$\min\left\{c_I + \frac{1}{N}\sum_{i\in I} \mu_i^{-\frac{N-2}{2}}S^\frac{N}{2}:|I|=1\right\}.$$
Without loss of generality, we may assume that $I=\{M\}$. By Lemma \ref{lem:bn_compactness} and our induction hypothesis, there exists a positive, least energy, fully nontrivial solution $(u_1,\ldots,u_{M-1})$ to the system $(\mathscr{S}_I)$. Fix $\xi\in\Omega$ and $\varrho\in (0,1)$ such that $B_\varrho(\xi)\subset\Omega$, and a cut-off function $\chi\in\mathcal{C}^\infty_c(B_\varrho(\xi))$ such that $0\leq\chi\leq 1$ and $\chi\equiv 1$ in $B_{\varrho/2}(\xi)$. Set 
$$w_\varepsilon(x):=\chi(x)\mu_M^\frac{2-N}{4} U_{\varepsilon,\xi},$$
where
$$U_{\varepsilon,\xi}:=a_N\left(\frac{\varepsilon}{\varepsilon^2+|x-\xi|^2}\right)^\frac{N-2}{2},\qquad\text{with }\;a_N=(N(N-2))^\frac{N-2}{4}.$$
It is shown in \cite{bn} that
\begin{align}
\|w_\varepsilon\|^2 &=\mu_M^\frac{2-N}{2}S^\frac{N}{2}+O(\varepsilon^{N-2}),\;\qquad |w_\varepsilon|^{2^*}_{2^*}=\mu_M^{-\frac{N}{2}}S^\frac{N}{2}+O(\varepsilon^N),\label{eq:bn1} \\
|w_\varepsilon|_2^2 &\geq
\begin{cases}
d_0\varepsilon^2+O(\varepsilon^{N-2})&\quad\text{if } N\geq 5,\\
d_0\varepsilon^2|\ln\varepsilon|+O(\varepsilon^2)&\quad\text{if } N=4,
\end{cases} \label{eq:bn2}
\end{align}
for some $d_0>0$; see also \cite[Lemma 1.46]{w}. Inspecting the proofs in \cite{bn, w} one sees that $d_0$ may be chosen independently of $\varrho$. Moreover, if $\alpha,\beta>1$ and $\alpha+\beta=2^*$, we have that
\begin{equation} \label{eq:bn3}
|w_\varepsilon|_\beta^\beta \leq c_1\int_{B_\varrho(0)}\left(\frac{\varepsilon}{\varepsilon^2+|x|^2}\right)^{\frac{N-2}{2}\beta}\mathrm{d}x\leq
\begin{cases} 
d_1\varepsilon^{\frac{N-2}{2}\beta} &\;\text{if }\beta\neq \frac{2^*}2,\\
d_1\varepsilon^\frac{N}{2}(1+|\ln\varepsilon|)  &\;\text{if }\beta =\frac{2^*}2,
\end{cases}
\end{equation}
for some $c_1, d_1>0$ and $\varepsilon$ small enough. By a regularity result in \cite[Appendix B]{s}, $u_i\in\mathcal{C}^0(\overline{\Omega})$ and we get that
$$\int_\Omega |w_\varepsilon|^\alpha|u_i|^\beta\leq \left(\max_{x\in\bar{\Omega}}|u_i(x)|^\beta\right)\int_\Omega|w_\varepsilon|^\alpha \to 0\quad\text{as }\varepsilon\to 0.$$
Hence, there exists $\varepsilon_0>0$ such that, for every $\varepsilon\in(0,\varepsilon_0)$,
\begin{align*}
&\mu_i|u_i|_{2^*}^{2^*}+\sum_{j\neq i}\beta_{ij}\,\lambda_{ij}\int_\Omega |u_j|^{\alpha_{ij}}|u_i|^{\beta_{ij}} + \beta_{iM}\,\lambda_{iM}\int_\Omega |w_\varepsilon|^{\alpha_{iM}}|u_i|^{\beta_{iM}}\\
&\qquad\qquad=\|u_i\|_i^2 + \beta_{iM}\,\lambda_{iM}\int_\Omega |w_\varepsilon|^{\alpha_{iM}}|u_i|^{\beta_{iM}} >0,\qquad i,j=1,\ldots,M-1, \\
&\mu_M|w_\varepsilon|_{2^*}^{2^*}+\sum_{j=1}^{M-1}\beta_{Mj}\,\lambda_{Mj}\int_\Omega |u_j|^{\alpha_{Mj}}|w_\varepsilon|^{\beta_{Mj}} > 0.
\end{align*}
Therefore we may use Lemma \ref{lem:maximality} in order to obtain $0<r<R<\infty$ and $s_{\varepsilon,1},\ldots,s_{\varepsilon, M}\in[r,R]$ such that
$$u_\varepsilon =(s_{\varepsilon,1}u_1,\ldots,s_{\varepsilon,M-1}u_{M-1},s_{\varepsilon,M}w_\varepsilon)\in\mathcal{N}.$$
As $(u_1,\ldots,u_{M-1})$ is a least energy solution to the system $(\mathscr{S}_I)$, from \eqref{eq:maximum} and estimates \eqref{eq:bn1} we obtain
\begin{align*}
\mathcal{J}(u_\varepsilon) &= \frac{1}{2}\sum_{i=1}^{M-1}s_{\varepsilon,i}^2\|u_i\|_i^2 - \frac{1}{2^*}\sum_{i=1}^{M-1}s_{\varepsilon,i}^{2^*}\mu_i|u_i|^{2^*}_{2^*} \\
&\quad- \frac{1}{2}\sum_{\substack{j\neq i\\i,j\ne M}}s_{\varepsilon,j}^{\alpha_{ij}} s_{\varepsilon,i}^{\beta_{ij}}\lambda_{ij}\int_\Omega|u_j|^{\alpha_{ij}}|u_i|^{\beta_{ij}}+\frac{1}{2}s_{\varepsilon,M}^2\|w_\varepsilon\|^2 - \frac{1}{2^*}s_{\varepsilon,M}^{2^*}\mu_M|w_\varepsilon|^{2^*}_{2^*} \\
&\quad+\frac{1}{2}s_{\varepsilon,M}^2\kappa_M|w_\varepsilon|_2^2 - \sum_{i=1}^{M-1}s_{\varepsilon,M}^{\alpha_{iM}} s_{\varepsilon,i}^{\beta_{iM}}\lambda_{iM}\int_\Omega|w_\varepsilon|^{\alpha_{iM}}|u_i|^{\beta_{iM}} \\
&\leq c_I + \frac{1}{N}\mu_M^{-\frac{N-2}{2}}S^\frac{N}{2} + O(\varepsilon^{N-2})\\
&\quad - \frac{1}{2}r^2|\kappa_M||w_\varepsilon|_2^2 + \sum_{i=1}^M R^{2^*}|\lambda_{iM}|\int_\Omega|w_\varepsilon|^{\alpha_{iM}}|u_i|^{\beta_{iM}}.
\end{align*}
So, if either $N\geq 6$, or $N=5$ and $\min\{\alpha_{ij},\beta_{ij}\}>\frac{4}{3}$ for all $i,j=1,\ldots,M$, we derive from \eqref{eq:bn2} and \eqref{eq:bn3} that, for $\varepsilon$ small enough,
$$- \frac{1}{2}r^2|\kappa_M||w_\varepsilon|_2^2 + \sum_{i=1}^M R^{2^*}|\lambda_{iM}|\int_\Omega|w_\varepsilon|^{\alpha_{iM}}|u_i|^{\beta_{iM}} \leq - C\varepsilon^2 + o(\varepsilon^2)$$
and \eqref{eq:bn_estimate} follows.
In the remaining cases we need to be careful when selecting $\xi$ and $\varrho$. If $N=4$ and $\alpha_{ij}=\beta_{ij}=2$ for all $i,j=1,\ldots,M$, we choose them in such a way that
$$\max_{x\in B_\varrho(\xi)}|u_i(x)|^2\leq \frac{r^2|\kappa_M|}{4M R^4|\lambda_{iM}|}\qquad\text{for every } i=1,\ldots,M-1.$$
This can be done because $u_i=0$ on $\partial\Omega$. Then, from \eqref{eq:bn2} we get
\begin{align*}
- \frac{1}{2}r^2|\kappa_M||w_\varepsilon|_2^2 + \sum_{i=1}^M R^4|\lambda_{iM}|\int_\Omega|w_\varepsilon|^2|u_i|^2 &\leq - \frac{1}{4}r^2|\kappa_M||w_\varepsilon|_2^2 \\
&\leq -C\varepsilon^2|\ln\varepsilon|+O(\varepsilon^2).
\end{align*}
If $N=5$ and $\min\{\alpha_{ij},\beta_{ij}\}=\frac{4}{3}$ for some pairs $i,j=1,\ldots,M$ we argue in a similar way. Hence, in all cases,  
$$\inf_{u\in\mathcal{N}}\mathcal{J}(u)\leq\mathcal{J}(u_\varepsilon)< c_I + \frac{1}{N}\mu_M^{-\frac{N-2}{2}}S^\frac{N}{2},$$
for $\varepsilon$ small enough, as claimed.
\end{proof}

\begin{proof}[Proof of Theorem \ref{thm:bn}]
The result follows from Lemmas \ref{lem:bn_compactness}, \ref{lem:bn_estimate} and  Theorem \ref{thm:B}$(a)$.
\end{proof}

\begin{remark} \label{rem} 
\emph{If $N=5$ and $\min\{\alpha_{ij},\beta_{ij}\}>\frac43$ or $N\ge 6$, then the condition $\partial\Omega\in \mathcal{C}^2$ is not necessary because $u_i\in \mathcal{C}^0(\Omega)$ according to the results in \cite[Appendix B]{s} and we may choose any $\xi,\varrho$ such that $\overline{B}_\varrho(\xi)\subset\Omega$.}
\end{remark}

\bigskip

\begin{flushleft}
\textbf{Mónica Clapp}\\
Instituto de Matemáticas\\
Universidad Nacional Autónoma de México\\
Circuito Exterior, Ciudad Universitaria\\
04510 Coyoacán, Ciudad de México, Mexico\\
\texttt{monica.clapp@im.unam.mx} 
\medskip

\textbf{Andrzej Szulkin}\\
Department of Mathematics\\
Stockholm University\\
106 91 Stockholm, Sweden\\
\texttt{andrzejs@math.su.se} 
\end{flushleft}

\end{document}